\documentclass[12pt]{article}

\input diagram.tex
\pdfoutput=1
\usepackage{amsmath} % eg. for \begin{equation}
\usepackage[mathscr]{eucal}
\usepackage{mathrsfs}
\usepackage{amssymb} % eg. for \leqslant
\usepackage{theorem}
\usepackage{rotating}
\usepackage{stmaryrd}
\usepackage[all,cmtip]{xy} % for \xymatrix
\usepackage{color}
\usepackage{faktor}

\theoremstyle{change} % puts numbers IN FRONT of "Theorem"
\newtheorem{theorem}{Theorem}[section] % defines environment "Theorem".
\newtheorem{lemma}[theorem]{Lemma} % defines environment "Lemma", that
\theorembodyfont{\rmfamily} % has the effect that the content of Remark,
\newtheorem{remark}[theorem]{Remark}

\newtheorem{definition}[theorem]{Definition}

\newtheorem{nothing}[theorem]{} % empty Theoremumgebung.

\newenvironment{proof}{\noindent{\bf Proof}\ }{\qed\bigskip}

\renewcommand{\le}{\leqslant}
\renewcommand{\ge}{\geqslant} % needs amssymb-Paket
\renewcommand{\leq}{\leqslant}
\renewcommand{\geq}{\geqslant}

\newcommand{\BIGOP}[1]
  {\mathop{\mathchoice
  {\raise-0.22em\hbox{\huge $#1$}}
  {\raise-0.05em\hbox{\Large $#1$}}{\hbox{\large $#1$}}{#1}}}
\newcommand{\bullA}{\bullet_A}

\newcommand{\catfont}{\mathsf}
\newcommand{\CC}{\mathbb{C}}

\newcommand{\Hom}{\mathrm{Hom}}

\newcommand{\hgt}{\mathrm{ht}}

\newcommand{\id}{\mathrm{id}}

\newcommand{\NN}{\mathbb{N}}

\newcommand{\supp}{\mathrm{supp}}

\newcommand{\lposet}[1]{\llap{\phantom{|}}_{#1}\catfont{poset}}
\newcommand{\lPoset}[1]{\llap{\phantom{|}}_{#1}\catfont{Poset}}

\newcommand{\lsimp}[1]{\llap{\phantom{|}}_{#1}\catfont{simp}}
\newcommand{\lSimp}[1]{\llap{\phantom{|}}_{#1}\catfont{Simp}}

\newcommand{\ltop}[1]{\llap{\phantom{|}}_{#1}\catfont{top}}
\newcommand{\lTop}[1]{\llap{\phantom{|}}_{#1}\catfont{Top}}
\newcommand{\poset}{\catfont{poset}}
\newcommand{\Poset}{\catfont{Poset}}

\newcommand{\qed}{\nobreak\hfill
                   \vbox{\hrule\hbox{\vrule\hbox to 5pt
                   {\vbox to 8pt{\vfil}\hfil}\vrule}\hrule}}

\newcommand{\RR}{\mathbb{R}}
\newcommand{\scrH}{\mathscr{H}}

\newcommand{\scrS}{\mathscr{S}}

\newcommand{\scrT}{\mathscr{T}}

\newcommand{\scrU}{\mathscr{U}}
\newcommand{\setA}{\catfont{set}^A}
\newcommand{\SetA}{\catfont{Set}^A}
\newcommand{\sigmabar}{\bar{\sigma}}

\newcommand{\simp}{\catfont{simp}}
\newcommand{\Simp}{\catfont{Simp}}
\newcommand{\Top}{\catfont{Top}}

\title{Monomial structures, I}{}%{footnote{{\bf MR Subject Classification:}  ??, ??{\bf Keywords:}  ??, ??.}}

\author{\small Robert Boltje\\
  \small Department of Mathematics\\
  \small University of California\\
  \small Santa Cruz, CA 95064\\
  \small U.S.A.\\
  \small boltje@ucsc.edu
  \and
  \small Hatice Mutlu\\
  \small Department of Mathematics\\ 
  \small University of California\\
  \small Los Angeles, CA 90095\\
  \small U.S.A.\\
  \small hmutlu@math.ucla.edu}
\date{October 17, 2023}

\begin{document}
\sloppy

%%%%%%%%%%%%%%%%% TITLE %%%%%%%%%%%%%%%%%%%%%%%%%%%%%%%%%%%%%%

\maketitle

%%%%%%%%%%%%%%%%% ABSTRACT %%%%%%%%%%%%%%%%%%%%%%%%%%%%%%%%%%%

\begin{abstract}
The goal of a series of papers is to define $G$-actions on various $A$-fibered structures, where $G$ is a finite group and $A$ is an abelian group. One prominent such example is the $A$-fibered Burnside ring. If $A=\CC^\times$, it is also called the ring of monomial representations (introduced by Dress in \cite{Dress1971}) and is the natural home for the canonical induction formula  (see \cite{Boltje1990}). In this first part of the series, motivated by constructions in \cite{BoucMutlu}, we introduce $A$-fibered structures on posets, on abstract simplicial complexes, and on $A$-bundles over topological spaces, together with natural notions of homotopy, and functors between these structures respecting homotopy. In a sequel we will continue with $G$-representations in these $A$-fibered structures and associate to them elements in the $A$-fibered Burnside ring.
\end{abstract}

%%%%%%%%%%%%%% SECTION 1 %%%%%%%%%%%%%%%%%%%%%%%%%%%%%%%%%%%%%%

\section{Introduction}\label{sec intro}

Actions of a finite group $G$ on sets, vector spaces, topological spaces, bundles, partially ordered sets (posets for short), simplicial complexes, and other mathematical objects are commonly known concepts. Another such occurrence is an action of $G$ on the category of free $A$-sets for an abelian group $A$, which leads to the notion of an $A$-fibered $G$-set, or {\em $A$-monomial $G$-set}. To our knowledge, this concept was first introduced by Dress in \cite{Dress1971}. Dress studied the Grothendieck group of $A$-monomial $G$-sets with finitely many $A$-orbits as a generalization of, the {\em $A$-fibered Burnside ring} $\Omega(G,A)$ of $G$, as a generalization of the Burnside ring $\Omega(G)=\Omega(G,\{1\})$, where $A$ is the trivial group. The $\CC^\times$-fibered Burnside ring of $G$ and its functorial properties played also a key role in the canoncial Brauer induction formula, which takes values in $\Omega(G,A)$, see \cite{Boltje1990}.

More recently, Bouc and Mutlu introduced in \cite{BoucMutlu} the concept of $A$-monomial structures on $G$-posets. There it is shown that each $A$-monomial $G$-poset determines an element in $\Omega(G,A)$ and that each element in $\Omega(G,A)$ is associated to some $A$-monomial $G$-poset. Closely related to posets are abstract simplicial complexes and also topological spaces. Our eventual goal is to also introduce $A$-monomial $G$-simplicial complexes and $A$-bundles with $G$-action, since also they have invariants in $\Omega(G,A)$. Symonds' geometric approach to the canonical Brauer induction formula, for instance associates to the tautological $\CC^\times$-bundle on the projecive space of a $\CC$-vector space with linear $G$-action and element in $\Omega(G,A)$, see \cite{Symonds}. This raises the question if there is a categorification of the canonical induction formula that takes values in some category of discrete objects, like $A$-monomial $G$-posets, or better homotopy classes of such. This is the motivation for eventually also introducing natural notions of homotopy on the $A$-monomial $G$-posets, $A$-monomial $G$-simplicial complexes and $A$-bundles with $G$-action.

In the case of sets or posets, the $G$-action has been considered first and the $A$-monomial structure was added later as a generalization (of the case $A=\{1\}$). For efficiency, we will reverse this order. In this paper, we will introduce first $A$-monomial structures on posets, simplicial complexes and topological spaces, together with notions of homotopy and with functors between them that respect homotopy. In a follow-up paper we will introduce $A$-monomial $G$-structures by means of representations of $G$ in the various categories of $A$-monomial structures. Viewing such $G$-structures as functors form the one-object category $G$ to the various $A$-monomial categories, turns out to be more effective for proving the necessary properties. 

The paper is arranged as follows. In Section 2, 3, and 4, we introduce $A$-monomial posets, $A$-monomial simplicial complexes, and $A$-fibered bundles, respectively, together with natural notions of homotopy between them. In Section 5, we define functors between these three categories and show that they preserve homotopy.

%%%%%%%%%%%%%% SECTION 2 %%%%%%%%%%%%%%%%%%%%%%%%%%%%%%%%%%%%%%

\section{$A$-monomial posets}

Throughout this section, $A$ denotes a group. %We recall the definition of the category of $A$-monomial $G$-posets from \cite{BM} when the finite group $G$ is the trivial group and define a notion of {\em homotopy} for morphisms of this category.

\begin{nothing}\label{noth PosetA}
{\em The category $\lPoset{}^A$.}\quad An {\em $A$-monomial poset} is a pair $(X,l)$, consisting of a poset $X$ viewed as a category and a functor $\l\colon X \to \bullA$, where $\bullA$ denotes the category with a single object $\bullet$ and $\Hom_{\bullA}(\bullet,\bullet):=A$, composition given by multiplication in $A$. The functor $l$ is 
given by a family of elements $l(x,x')\in A$, one for each pair $(x,x')\in X\times X$ with $x\le x'$, satisfying
\begin{equation}\label{eqn l conditions poset}
l(x', x'') \cdot l(x,x') = l( x, x'')\quad \text{and} \quad l(x,x)=1_A\,,
\end{equation} 
whenever $x\le x'\le x''$ in $X$.

Let $(X,l)$ and $(Y,m)$ be two $A$-monomial posets. A {\em morphism} from $(X,l)$ to $(Y,m)$ is a pair $(f,\lambda)$ consisting of a morphism $f\colon X\to Y$ of posets and a natural transformation $\lambda\colon l\to m\circ f$ of functors $X\to\bullA$. The datum of the natural transformation $\lambda$ is equivalent to a function $\lambda\colon X\to A$ satisfying
\begin{equation}\label{eqn lambda conditions poset}
m(f(x),f(x')) \cdot \lambda(x) = \lambda(x') \cdot l(x,x')
\end{equation}
for every  $x,x'\in X$ with $x\le x'$. If also $(e,\mu)\colon (Y,m)\to(Z,n)$ is a morphism in $\lPoset{}^A$ then their composition is defined by 
\begin{equation*}
(e,\mu)\circ(f,\lambda):=(ef,(\mu*f)\circ \lambda)\,,
\end{equation*}
where $\mu*f\colon mf\to nef$ is the composition of the natural transformations $\mu\colon m\to ne$ and $\id_{f}\colon f\to f$, where $f\colon X\to Y$ is viewed as a functor. Note that the function $X\to A$ associated to $(\mu*f)\circ\lambda$ is given by \begin{equation*}
   x\mapsto \mu(f(x))\cdot\lambda(x)\,.
\end{equation*}

The $A$-monomial posets with the above morphisms define a category $\lPoset{}^A$. The identity morphism of $(X,l)$ is the morphism $(\id_X,1)$, where $1(x,x'):=1_A$ whenever $x\le x'$ in $X$. The full subcategory of finite $A$-monomial posets $(X,l)$ (i.e., where $X$ is finite) is denoted by $\lposet{}^A$.
\end{nothing}

\begin{remark}\label{rem SetA}
The category $\SetA$ (resp.~$\setA$) of free left $A$-sets (resp.~those with finitely many $A$-orbits) is equivalent to the full subcategory of $\Poset^A$ (resp.~$\poset^A$) consisting of objects $(X,l)$, where $X$ is a {\em discrete poset}, i.e., $x\le x'$ implies $x=x'$. In fact, one obtains a functor on this subcategory by sending the object $(X,l)$ to the free $A$-set $A\times X$ with $A$-action given by $a(b,x):=(ab,x)$ and by sending a morphism $(f,\lambda)\colon (X,l)\to(Y,m)$ to the $A$-equivariant map $A\times X\to A\times Y$, $(a,x)\mapsto (a\lambda(a)^{-1}, f(x))$. This functor is faithful and essentially surjective. In this sense, $\Poset^A$ generalizes the notion of \lq free $A$-bundles on sets\rq.
\end{remark}

\begin{nothing}{\em The partial order $(f,\lambda)\le (f',\lambda')$.} Let $(f,\lambda),(f',\lambda')\colon (X,l)\to (Y,m)$ be morphisms in $\lPoset{}^A$. We write $f\le f'$ if $f(x)\le f'(x)$ for all $x\in X$. The relation $f\le f'$ gives rise to a natural transformation
\begin{equation*}
\iota:=\iota_{f\le f',m}\colon m\circ f\to m\circ f'
\end{equation*}
between functors $X\to\bullA$, given by $\iota(x):=m(f(x),f'(x))$, for every $x\in X$. This is in fact a natural transformation, since,
for any $x,x'\in X$ with $x\le x'$, we have
\begin{equation*}
m(f'(x),f'(x'))\cdot m(f(x),f'(x)) = m(f(x),f'(x')) = m(f(x'), f'(x'))\cdot m(f(x), f(x'))\,.
\end{equation*}
Now we define
\begin{equation*}
(f,\lambda)\le (f',\lambda') :\iff f\le f'\quad\text{and}\quad \iota_{f\le f',m}\circ \lambda = \lambda'\,.
\end{equation*}
Note that the latter equality is equivalent to 
\begin{equation}\label{eqn le condition in poset}
m(f(x),f'(x)) \cdot \lambda(x) = \lambda'(x)\,,
\end{equation}
for all $x\in X$.
It is straightforward to verify that this defines a partial order on the set $\Hom_{\lPoset{}^A}((X,l),(Y,m))$.	
\end{nothing}

\begin{definition}	
We call two morphisms $(f,\lambda),(f',\lambda')\colon (X,l)\to(Y,m)$ in $\lPoset{}^A$ {\em comparable} and write $(f,\lambda)-(f',\lambda')$ if $(f,\lambda)\le (f',\lambda')$ or $(f',\lambda')\le (f,\lambda)$. The relation $-$ is reflexive and symmetric. We denote its transitive closure by $\sim$, i.e., $(f,\lambda)\sim(f,\lambda')$ if there exist $(f_0,\lambda_0),\ldots,(f_n,\lambda_n)\colon (X,l)\to (Y,m)$ in $\lPoset{}^A$ such that
\begin{equation*}
	(f,\lambda)=(f_0,\lambda_0) - (f_1,\lambda_1) - \cdots - (f_n,\lambda_n) = (f',\lambda')\,.
\end{equation*}
In this case we call the two morphisms $(f,\lambda)$ and $(f',\lambda')$ {\em homotopic}. Thus, $(f,\lambda)$ and $(f',\lambda')$ are homotopic if and only if they belong to the same connected component of the poset $\Hom_{\Poset^A}((X,l),(Y,m))$.
\end{definition}

\begin{lemma}\label{lem homotopy for posets}
Let $(f,\lambda),(f',\lambda')\colon (X,l)\to(Y,m)$ and $(e,\mu),(e',\mu')\colon (Y,m)\to (Z,n)$ be morphisms in $\lPoset{}^A$. If $(f,\lambda)\sim(f',\lambda')$ and $(e,\mu)\sim(e',\mu')$ then $(e,\mu)\circ(f,\lambda)\sim(e',\mu')\circ(f',\lambda')$.
\end{lemma}

\begin{proof}
Since $\sim$ is is an equivalence relation, it suffices by symmetry and transitivity, to show that $(f,\lambda)\le (f',\lambda')$ implies $(e,\mu)\circ(f,\lambda)\le (e,\mu)\circ(f',\lambda')$ and that $(e,\mu)\le (e',\mu')$ implies $(e,\mu)\circ(f',\lambda')\le (e',\mu')\circ (f',\lambda')$. Both implications are straightforward verifications.
\end{proof}

The previous lemma allows us to define the homotopy category of $\Poset^A$.

\begin{definition}\label{def homotopy cat for posets}
We define the {\em homotopy category} $\scrH(\lPoset{}^A)$ of $\lPoset{}^A$ as the category whose objects are the same as those of $\lPoset{}^A$ and whose morphisms are homotopy classes $[f,\lambda]$ of morphisms $(f,\lambda)$ in $\lPoset{}^A$. 
Lemma~\ref{lem homotopy for posets} implies that the composition in $\lPoset{}^A$ induces a composition of homotopy classes of morphisms. 
By $\scrH(\lposet{}^A)$ we denote the full subcategory of $\scrH(\lPoset{}^A)$ whose objects are finite $A$-monomial posets. 
We have obvious functors $\lPoset{}^A\to\scrH(\lPoset{}^A)$ and $\lposet{}^A\to\scrH(\lposet{}^A)$.
\end{definition}

\begin{remark} If $A=\{1\}$ then the categories $\Poset^A$ and $\Poset$ are isomorphic via $(X,l)\mapsto X$.
Thus, the above definition yields in particular intrinsic notions of homotopy and a homotopy category $\scrH(\Poset)$. A similar statement holds for the finite versions.
\end{remark}

\begin{nothing}\label{noth refining <=} {\em Refining $(f,\lambda)\le(f',\lambda')$.}
The following set-up and construction will be used in Lemma~\ref{lem Sigma and homotopy}. 

\smallskip
(a) Assume that $X$ is a finite poset. Define the {\em height} $\hgt(x)$ of an element $x\in X$ as the maximum $i\in\NN_0$ such that there exists a strictly ascending chain $x_0<\ldots<x_i$ in $X$ with $x_i=x$. Note that if $x_0<\ldots<x_k$ is a strictly ascending chain in $X$ then 
\begin{equation*}
\hgt(x_0)<\ldots<\hgt(x_k)
\end{equation*}
and therefore $\{x_0,\ldots,x_k\}$ cannot have more than one element of a given height.

\smallskip
(b) Assume now that $X$ and $Y$ are posets, that $X$ is finite, and that $f\le f'$ in $\Hom_{\lPoset{}}(X,Y)$. Let $n\in\NN_0$ be the maximal height of all elements of $X$. With the notation from Part~(a) we define, for $i\in\{0,\ldots,n+1\}$, a function $f_i\colon X\to Y$ by
\begin{equation*}
f_i(x):= \begin{cases} f(x) & \text{if $\hgt(x) \le n-i$,} \\
f'(x) & \text{if $\hgt(x)> n-i$.}
\end{cases}
\end{equation*}
Then $f_0,\ldots,f_{n+1}$ are again morphisms of posets satisfying
\begin{equation*}
f=f_0\le \cdots \le f_{n+1} = f'\,.
\end{equation*}
In fact, for $x\in X$ and $i\in \{0,\ldots,n\}$, one has
\begin{align}\label{eqn f i explicit}
&  f_i(x) = f(x) = f_{i+1}(x) & \text{if $\hgt(x)<n-i$,}\notag\\
& f_i(x)=f(x)\le f'(x)=f_{i+1}(x) & \text{if $\hgt(x)=n-i$,} \\
& f_i(x)=f'(x)=f_{i+1}(x) & \text{if $\hgt(x)>n-i$.}\notag
\end{align}

\smallskip
(c) Finally, assume that $(X,l)$ and $(Y,m)$ are $A$-monomial posets, that $X$ is finite, and that $(f,\lambda)\le (f',\lambda')$ in $\Hom_{\lPoset{}^A}((X,l),(Y,m))$. With the above notation we define, for $i\in\{0,\ldots,n+1\}$, a function $\lambda_i\colon X\to A$ by
\begin{equation*}
     \lambda_i(x):= 
     \begin{cases} 
         \lambda(x) & \text{if $\hgt(x) \le n-i$,} \\ 
         \lambda'(x) & \text{if $\hgt(x)> n-i$.}
     \end{cases}
\end{equation*}
We claim that 
\begin{equation}\label{eqn refined morphisms}
     (f_0,\lambda_0),\ldots,(f_{n+1}, \lambda_{n+1})\in\Hom_{\lPoset{}^A}((X,l),(Y,m))
\end{equation} 
and
\begin{equation}\label{eqn refined <=}
     (f,\lambda)=(f_0,\lambda_0) \le \cdots \le (f_{n+1}, \lambda_{n+1}) = (f',\lambda')\,.
\end{equation}
In fact, for $x\in X$ and $i\in\{0,\ldots,n\}$, one has
\begin{align}\label{eqn lambda i explicit}
&  \lambda_i(x) = \lambda_{i+1}(x) = \lambda(x) & \text{if $\hgt(x)<n-i$,}\notag\\
& \lambda_i(x)=\lambda(x),\ \ \lambda_{i+1}(x)=\lambda'(x)  & \text{if $\hgt(x)=n-i$,}\\
& \lambda_i(x)=\lambda_{i+1}(x) = \lambda'(x) & \text{if $\hgt(x)>n-i$,}\notag
\end{align}
Now, for all $i\in\{0,\ldots,n+1\}$ and all $x,x'\in X$ satisfying $x\le x'$, the equation
\begin{equation*}
m(f_i(x),f_i(x')) \cdot \lambda_i(x) = \lambda_i(x') \cdot l(x,x')
\end{equation*}
holds by distinguishing the cases (i) $\hgt(x)\le\hgt(x')\le n-i$, (ii) $\hgt(x)\le n-i<\hgt(x')$, (iii) $n-i<\hgt(x)\le\hgt(x')$. Case~(i) follows from $(f,\lambda)$ being a morphism, Case (iii) follows from $(f',\lambda')$ being a morphism, and Case (ii) follows from
\begin{align*}
m(f(x),f'(x'))\cdot\lambda(x) 
& = m(f(x'), f'(x'))\cdot m(f(x), f(x')) \cdot \lambda(x)\\
& = m(f(x'), f'(x'))\cdot \lambda(x') \cdot l(x,x') = \lambda'(x') \cdot l(x,x')\,,
\end{align*}
since $(f,\lambda)\le(f',\lambda')$. Thus, the statement in (\ref{eqn refined morphisms}) holds. To verify the statement in (\ref{eqn refined <=}), note that clearly $(f_0,\lambda_0)=(f,\lambda)$ and $(f_{n+1},\lambda_{n+1})=(f',\lambda')$. Moreover, the equation
\begin{equation*}
m(f_i(x),f_{i+1}(x)) \cdot \lambda_i(x) = \lambda_{i+1}(x)
\end{equation*}
holds for all $i\in\{0,\ldots,n\}$ and $x\in X$ by distinguishing the cases (i) $\hgt(x)<n-i$, (ii) $\hgt(x)=n-i$, and (iii) $\hgt(x)>n-i$. In fact, in case (i) we use $m(f(x),f(x))=1$, in case (iii) we use $m(f'(x),f'(x))=1$, and in case (ii) we use that $(f,\lambda)\le (f',\lambda')$.
\end{nothing}

%%%%%%%%%%%%%% SECTION 3 %%%%%%%%%%%%%%%%%%%%%%%%%%%%%%%%%%%%%%

\section{$A$-monomial simplicial complexes}
Throughout this section,  $A$ denotes again a group. In this section we define the category of $A$-monomial simplicial complexes $\lSimp{}^A$ and its homotopy category $\scrH(\lSimp{}^A)$.

\begin{nothing}{\em The category $\lSimp{}^A$.}\quad A {\em simplicial complex} is a pair $(S, \mathscr{S})$ consisting of a set $S$ and 
a set $\mathscr{S}$ of nonempty finite subsets of $S$ such that 
$\{s\} \in \mathscr{S}$, for every $s\in S$, and such that $\sigma\in\scrS$ and $\emptyset\neq\sigma'\subseteq\sigma$ implies $\sigma'\in\scrS$.
The elements $\sigma$ of $\scrS$ are called the {\em simplices} of $(S,\scrS)$. 
A {\em morphism} (or {\em simplicial map}) between two simplicial complexes $(S,\scrS)$ and $(T,\scrT)$, is a map $f\colon S\to T$ such that $f(\sigma)\in\scrT$ for every $\sigma\in\scrS$.
The simplicial complexes and their morphisms form a category that we denote by $\lSimp{}$. 
The full subcategory of object $(S,\scrS)$ where $S$ is finite is denoted by $\lsimp{}$. 

Given a simplicial complex $(S, \mathscr{S})$, the set of simplices $\scrS$ can be viewed as a partially ordered set via inclusion of subsets of S.

An {\em $A$-monomial simplicial complex} is a triple $(S, \mathscr{S}, l)$ consisting 
of a simplicial complex $(S, \mathscr{S})$ and a functor $l\colon \mathscr{S} \rightarrow \bullA$ where the poset $\scrS$ is viewed as a category. The datum of the functor $l$ is equivalent to a family of elements $l(\sigma,\sigma')\in A$, one for each pair $(\sigma,\sigma')\in \scrS\times\scrS$ with $\sigma\subseteq\sigma'$, satisfying
\begin{equation}\label{eqn l conditions simp}
   l(\sigma',\sigma'') \cdot l(\sigma,\sigma') = l(\sigma,\sigma'')\quad\text{and}\quad l(\sigma,\sigma)=1_A
\end{equation}
whenever $\sigma\subseteq\sigma'$ and $\sigma'\subseteq\sigma''$ for every $\sigma,\sigma',\sigma''\in\scrS$.

Let $(S,\mathscr{S}, l)$ and $(T,\mathscr{T}, m)$ be two $A$-monomial simplicial complexes. 
A {\em morphism} (or {\em $A$-monomial simplicial map}) from $(S,\mathscr{S}, l)$ to $(T, \mathscr{T}, m)$ is a pair $(f, \lambda)$ consisting of a simplicial map $f\colon (S,\scrS) \to (T,\scrT)$ and a natural transformation $\lambda\colon l \rightarrow m\circ f$.
Note that the datum of the natural transformation $\lambda$ is equivalent to a function $\lambda\colon\scrS\to A$ satisfying
\begin{equation}\label{eqn lambda conditions simp}
   m(f(\sigma),f(\sigma'))\cdot \lambda(\sigma) = \lambda(\sigma')\cdot l(\sigma,\sigma')
\end{equation}
for every $\sigma,\sigma'\in\scrS$ with $\sigma\subseteq \sigma'$.
If also $(e, \mu)\colon (T, \mathscr{T}, m)\to (U, \mathscr{U}, n)$ is a morphism between $A$-monomial simplicial complexes then the composition of $(e, \mu)$ and $(f, \lambda)$ is given by 
$(e, \mu)\circ(f,\lambda):=(ef,(\mu*f)\circ \lambda)$, with $\mu*f\colon mf\to nef$ as in \ref{noth PosetA}. Note that the function $\scrS\to A$ associated to the natural transformation $(\mu*f)\circ \lambda$ is given by $\sigma\mapsto \mu(f(\sigma))\cdot\lambda(\sigma)$.

The category of $A$-monomial simplicial complexes and their morphisms will be denoted by $\lSimp{}^A$. The identity morphism of $(S,\scrS,l)$ is the morphism $(\id_S,1)$, where $1(\sigma,\sigma'):=1_A$ whenever $\sigma\subseteq\sigma'$. We will denote the full subcategory of finite $A$-monomial simplicial complexes by $\lsimp{}^A$.
\end{nothing}

\begin{nothing}{\em The relation $(f,\lambda) - (f',\lambda')$.} Let $f, f'\colon (S,\scrS)\to(T,\scrT)$ be morphisms in $\lSimp{}$. Following \cite[Chapter~3, Section~5]{Spanier}, we say that $f$ and $f'$ are {\em contiguous} and write $f-f'$ if $f(\sigma)\cup f'(\sigma)\in\scrT$ for every $\sigma\in\scrS$.
	
Now let $(f,\lambda),(f',\lambda')\colon (S,\scrS,l)\to (T,\scrT, m)$ be morphisms in $\lSimp{}^A$ such that the simplicial maps $f$ and $f'$ are contiguous. Then, this gives rise to a natural transformation
\begin{equation*}
\iota:= \iota_{f-f',m}\colon m\circ f\to m\circ f'
\end{equation*}
between functors $\scrS\to \bullA$, given by
\begin{equation*}
\iota(\sigma):= m(f'(\sigma),f(\sigma)\cup f'(\sigma))^{-1}\cdot m(f(\sigma), f(\sigma)\cup f'(\sigma))\,,
\end{equation*}
for $\sigma\in\scrS$. This is in fact a natural transformation, since, for any $\sigma,\sigma'\in\scrS$ with $\sigma\subseteq \sigma'$, we have
\begin{align*}
   & m(f'(\sigma),f'(\sigma'))\cdot \iota(\sigma) \\
   =\ & m(f'(\sigma),f'(\sigma')) \cdot m(f'(\sigma), f(\sigma)\cup f'(\sigma))^{-1} \cdot 
   m(f(\sigma), f(\sigma)\cup f'(\sigma)) \\
   =\ & m(f'(\sigma'), f(\sigma')\cup f'(\sigma'))^{-1}\cdot m(f(\sigma)\cup f'(\sigma), f(\sigma')\cup f'(\sigma')) \cdot
   m(f(\sigma), f(\sigma)\cup f'(\sigma)) \\
   =\ & m(f'(\sigma'), f(\sigma')\cup f'(\sigma'))^{-1}\cdot m(f(\sigma), f(\sigma')\cup f'(\sigma')) \\
   =\ & m(f'(\sigma'), f(\sigma')\cup f'(\sigma'))^{-1}\cdot m(f(\sigma'), f(\sigma')\cup f'(\sigma')) \cdot
   m(f(\sigma), f(\sigma')) \\
   =\ & \iota(\sigma') \cdot m(f(\sigma), f(\sigma'))\,.
\end{align*}

We now define the {\em contiguity} relation on $\Hom_{\lSimp{}^A}((S,\scrS,l), (T,\scrT,m))$ by
\begin{equation*}
(f,\lambda) - (f',\lambda) :\iff f-f' \text{ in } \lSimp{}\quad \text{and} \quad \iota_{f-f',m}\circ\lambda = \lambda'\,.
\end{equation*}
Note that the latter equation is equivalent to
\begin{equation}\label{eqn - condition in simp}
m(f(\sigma),f(\sigma)\cup f'(\sigma))\cdot \lambda(\sigma) =  m(f'(\sigma), f(\sigma)\cup f'(\sigma))\cdot \lambda'(\sigma)\,,
\end{equation}
for all $\sigma\in\scrS$.

It is straightforward to verify that the contiguity relation on $\Hom_{\lSimp{}^A}((S,\scrS,l),(T,\scrT,m))$ is symmetric and reflexive.
\end{nothing}

\begin{lemma}\label{lem iota properties simp}
Let $(f,\lambda),(f',\lambda') \colon (S,\scrS,l)\to (T,\scrT,m)$ and $(e,\mu), (e',\mu')\colon (T,\scrT, m)\to (U,\scrU,n)$ be morphisms in $\lSimp{}^A$.
	
\smallskip
{\rm (a)} If $f-f'$ in $\lSimp{}$ then $ef-ef'$ in $\lSimp{}$ and
\begin{equation*}
	\iota_{ef-ef',n} \circ (\mu*f) = (\mu*f') \circ \iota_{f-f',m}\,.	
\end{equation*}
	
\smallskip
{\rm (b)} If $e-e'$ in $\lSimp{}$ then $ef-e'f$ is $\lSimp{}$ and 
	\begin{equation*}
	\iota_{ef-e'f,n} = \iota_{e-e',n}*f\,.
\end{equation*}
	
\smallskip
{\rm (c)} If $(f,\lambda)-(f',\lambda')$ and $(e,\mu)-(e',\mu')$ then $(e,\mu)\circ (f,\lambda) - (e,\mu)\circ (f',\lambda')$ and $(e,\mu)\circ(f,\lambda) - (e',\mu')\circ(f,\lambda)$.
\end{lemma}

\begin{proof}
(a) Let $\sigma\in\scrS$ and set $\tau:=f(\sigma)\cup f'(\sigma)$. Then $\tau\in\scrT$, since $f-f'$, and we obtain $ef(\sigma)\cup ef'(\sigma) = e(\tau)\in\scrU$. Thus, $ef-ef'$. Moreover,
\begin{align*}
   (\iota_{ef-ef',n} \circ (\mu*f))(\sigma) & = n(ef'(\sigma), e(\tau))^{-1} \cdot n(ef(\sigma),e(\tau)) \cdot \mu(f(\sigma)) \\
   & = n(ef'(\sigma), e(\tau))^{-1} \cdot \mu(\tau) \cdot m(f(\sigma), \tau)\\
   & = \mu(f'(\sigma)) \cdot m(f'(\sigma),\tau)^{-1} \cdot m(f(\sigma), \tau) = ((\mu*f')\circ \iota_{f-f',m})(\sigma)\,.  
\end{align*}

\smallskip
(b) Let $\sigma\in\scrS$. Then $\tau:=ef(\sigma)\cup e'f(\sigma)\in\scrU$, since $f(\sigma)\in\scrT$ and $e-e'$. Thus, we get $ef-e'f$. Moreover, both $\iota_{ef-e'f,n}$ and $\iota_{e-e',n}*f$ evaluated at $\sigma\in\scrS$ are equal to $m(e'f(\sigma), \tau)^{-1}\cdot m(ef(\sigma), \tau)$.

\smallskip
(c) For the first part we need to show that $(ef, (\mu*f)\circ \lambda) - (ef', (\mu*f')\circ \lambda')$. But by Part~(a) we have $ef-ef'$ and 
\begin{equation*}
\iota_{ef-ef',n}\circ (\mu*f)\circ \lambda = (\mu*f') \circ \iota_{f-f',m} \circ \lambda = (\mu*f') \circ \lambda'\,.
\end{equation*}

\smallskip
For the second part we need to show that $(ef, (\mu*f)\circ \lambda) - (e'f, (\mu'*f)\circ \lambda)$. But, by Part~(b) we have $ef-e'f$ and 
\begin{equation*}
\iota_{ef-e'f,n} \circ (\mu*f)\circ \lambda  = (\iota_{e-e',n}*f)\circ (\mu*f)\circ \lambda 
= ((\iota_{e-e',n}\circ\mu)*f) \circ \lambda = (\mu'*f)\circ \lambda\,.
\end{equation*}
\end{proof}

\begin{definition}
Let $(S,\scrS,l)$ and $(T,\scrT,m)$ be $A$-monomial simplicial complexes. We denote the transitive closure of the relation $-$ on $\Hom_{\lSimp{}^A}((S,\scrS,l),(T,\scrT,m))$ by $\sim$ and call two morphisms $(f,\lambda), (f',\lambda')\in \Hom_{\lSimp{}^A}((S,\scrS,l),(T,\scrT,m))$ {\em homotopic} if $(f,\lambda)\sim (f',\lambda')$.
\end{definition}

The following Lemma follows now immediately from Lemma~\ref{lem iota properties simp}(c).

\begin{lemma}\label{lem homotopy for simp}
Let $(f,\lambda),(f',\lambda')\colon(S,\mathscr{S},l)\to(T,\mathscr{T},m)$ and $(e,\mu),(e',\mu')\colon(T,\mathscr{T},m)\to(R, \mathscr{R},n)$ be morphisms in $\lSimp{}^A$. If $(f,\lambda)\sim(f',\lambda')$ and $(e,\mu)\sim (e',\mu')$ then $(e,\mu)\circ (f,\lambda) \sim (e',\mu')\circ (f',\lambda')$.
\end{lemma}

Lemma~\ref{lem homotopy for simp} allows us to form the homotopy category of $\lSimp{}^A$.

\begin{definition}
We define the {\em homotopy category} $\scrH(\lSimp{}^A)$ of $\lSimp{}^A$ as the category whose objects are the same as those of $\lSimp{}^A$ and whose morphisms are homotopy classes $[f,\lambda]$ of morphisms $(f,\lambda)$ in $\lSimp{}^A$. Lemma~\ref{lem homotopy for simp} implies that the composition in $\lSimp{}^A$ induces a composition of homotopy classes of morphisms. By $\scrH(\lsimp{}^A)$ we denote the full subcategory of $\scrH(\lSimp{}^A)$ whose objects are finite $A$-monomial simplicial complexes. We obtain obvious functors $\lSimp{}^A\to \scrH(\lSimp{}^A)$ and $\lsimp{}^A\to\scrH(\lsimp{}^A)$.
\end{definition}

\begin{remark}
If $A=\{1\}$ then the categories $\lSimp{}^A$ and $\lSimp{}$ are isomorphic via $(S,\scrS, l)\mapsto (S,\scrS)$. 
Thus, the above definition yields in particular intrinsic notions of homotopy classes for $\Simp$ (which coincide with contiguity classes on $\Simp$ in \cite[Section 3.5]{Spanier}) and a homotopy category $\scrH(\Simp)$. A similar statement holds for the finite versions.
\end{remark}

%%%%%%%%%%%%%% SECTION 4 %%%%%%%%%%%%%%%%%%%%%%%%%%%%%%%%%%%%%%

\section{$A$-fibered bundles}
Let $A$ be a topological group. In this section we will define the category of $A$-fibered bundles $\lTop{}^A$ and its homotopy category $\scrH(\lTop{}^A)$. 

\begin{nothing}{\em The category $\lTop{}^A$.}\quad An $A$-fibered bundle is a triple $(\pi, E, X)$ where $E$ is a topological space on which $A$ acts via homeomorphisms, $X$ is a topological space equipped with the trivial $A$-action, and  $\pi \colon E \to X$ is a surjective $A$-equivariant continuous map, such that for every $x \in X$ there exists an open neighborhood $U$ of $x$ in $X$ and an $A$-equivariant homeomorphism $\varphi \colon \pi^{-1}(U) \to U \times A$ making the diagram

$$\xymatrix@R=4ex@C=3ex{
	\pi^{-1}(U)\ar[rr]^-\varphi\ar[rd]_-{\pi}&&U\times A\ar[ld]^{p_1}\\
	&U&
}
$$
commutative. Here, $A$ acts on $U\times A$ via $a(u,b):=(u,ab)$. In particular, the subspace $\pi^{-1}(\{x\})$ of $E$ is mapped homeomorphically and $A$-equivariantly onto $\{x\} \times A$. Therefore, $\pi^{-1}(\{x\})$ is $A$-stable and isomorphic to $A$ as a topological space and as an $A$-set. 

Let $\pi \colon E \to X$ and $\pi' \colon E' \to X'$ be two $A$-fibered bundles. A {\em morphism} from $\pi \colon E \to X$ to $\pi' \colon E' \to X'$ is a pair $(\varphi, f)$ consisting of a continuous map $f \colon X \to X'$ and an $A$-equivariant continuous map $\varphi \colon E \to E'$ such that the diagram
\begin{equation*}
\xymatrix@C=7ex@R=7ex{
	E\ar[r]^-{\varphi}\ar[d]_-{\pi}&E'\ar[d]^-{\pi'}\\
	X\ar[r]_-{f}& X'}
\end{equation*}
commutes.

The category of $A$-fibered bundles and their morphisms with the obvious composition
will be denoted by $\lTop{}^A$. The identity morphism of $\pi \colon E \to X$ is the morphism $(\id_E, \id_X)$.
For any topological space $X$, the map $p_X \colon A \times X \to X$, $(a,x) \mapsto x$, is called the {\em trivial $A$-fibered bundle on $X$}. \end{nothing}

Recall that two morphisms $f,f'\colon X \to X'$ in the category $\Top$ of topological space are called {\em homotopic}, denoted by $f\sim f'$, if there exists a morphism $h \colon X \times I \to X'$ in $\Top$ such that $h(x,0)=f(x)$ and $h(x,1)=f'(x)$ for all $x \in X$. Here $I=[0,1]$ denotes the unit interval.

\begin{definition}
Let  $(\varphi,f), (\varphi', f') \colon (\pi \colon E \to X) \to (\pi'\colon E' \to X')$ be morphisms in $\Top^A$. We say that $(\varphi , f)$ and $(\varphi', f')$ are {\em homotopic} and write $(\varphi,f) \sim (\varphi', f')$ if there exists an $A$-equivariant continuous map $H\colon E\times I \to E'$ and a continuous map $h\colon X\times I\to X'$ such that 
\begin{equation*}
\xymatrix@C=7ex@R=7ex{
	E\times I\ar[r]^-{H}\ar[d]_-{\pi\times 1}&E'\ar[d]^-{\pi'}\\
	X\times I \ar[r]_-{h}& X'}
\end{equation*}
commutes and $H(e,0)= \varphi(e)$, $H(e,1)=\varphi'(e)$, $h(x,0)=f(x)$,and $h(x,1)=f'(x)$ for all $e\in E$ and all $x\in X$. Here $A$ acts trivially on $I$.
\end{definition}

It is straightforward to verify that the homotopy relation on
$\Hom_{\lTop{}^A}((\pi\colon E \to X), (\pi'\colon E'\to X'))$ is symmetric, reflexive and transitive and that the following Lemma holds.

\begin{lemma}\label{lem homotopy for top}
Let  $(\varphi,f), (\varphi', f') \colon (\pi \colon E \to X) \to (\pi'\colon E' \to X')$ and $(\psi, e), (\psi',e') \colon (\pi'\colon E' \to X') \to (\pi''\colon E'' \to X'')$ be morphisms in $\Top^A$. If $(\varphi,f)\sim(\varphi',f')$ and $(\psi,e)\sim(\psi',e')$ then $(\psi,e)\circ (\varphi,f) \sim (\psi', e')\circ (\varphi',f')$.	
\end{lemma}

Lemma~\ref{lem homotopy for top} allows us to form the homotopy category of $\Top^A$.

\begin{definition}
We define the {\em homotopy category} $\scrH(\lTop{}^A)$ of $\lTop{}^A$ as the category whose objects are the same as those of $\lTop{}^A$ and whose morphisms are homotopy classes $[\varphi,f]$ of morphisms $(\varphi, f)$ in $\lTop{}^A$. Lemma~\ref{lem homotopy for top} implies that the composition in $\lTop{}^A$ induces a composition of homotopy classes of morphisms. 
\end{definition}

%%%%%%%%%%%%%% SECTION 5 %%%%%%%%%%%%%%%%%%%%%%%%%%%%%%%%%%%%%%

\section{Functors between $\Poset^A$, $\Simp^A$ and $\Top^A$}
In this section we introduce the functors $\Sigma\colon \lPoset{}^A\to\lSimp{}^A$, $\Pi\colon\lSimp{}^A\to\lPoset{}^A$ and $\mid \cdot \mid \colon \lSimp{}^A \to \lTop{}^A$.

\begin{nothing}{\em The functor $\Sigma\colon \Poset^A\to\Simp^A$.}\quad For any $A$-monomial poset $(X, l)$, we set $\Sigma(X,l)= (X,\Sigma(X), \Sigma(l))$ where $\Sigma(X)$ is the set of nonempty totally ordered finite subsets of $X$, i.e., strictly ascending chains $x_0<\ldots< x_n$ with $n\ge 0$, and where $\Sigma(l): \Sigma(X) \to \bullA$ is the functor given by the elements
\begin{equation*}
(\Sigma(l))(\sigma,\sigma'):= l(\sigmabar,\bar{\sigma'})\in A\,,
\end{equation*}
for any two chains $\sigma, \sigma' \in \Sigma(X)$ with $\sigma\subseteq\sigma'$, where $\sigmabar$ denotes the largest element of $\sigma$. This is well-defined, since $\sigma\subseteq\sigma'$ implies $\sigmabar\le \bar{\sigma'}$. It is now straightforward to verify Equation~(\ref{eqn l conditions simp}), so that $(X,\Sigma(X), \Sigma(l))$ is an $A$-monomial simplicial complex.

Given  a morphism $(f, \lambda)\colon (X,l) \to (Y,m)$   in $\Poset^A$, we define a morphism
$$\Sigma(f, \lambda):=(f,\Sigma(\lambda))\colon \Sigma(X,l)\to\Sigma(Y,m)$$
noting that $f\colon (X,\Sigma(X))\to(Y,\Sigma(Y))$ is a simplicial map and defining the natural transformation $\Sigma(\lambda)$ by the function $\Sigma(\lambda)\colon \Sigma(X)\to A$, $\sigma\mapsto\lambda(\sigmabar)$. This definition satisfies Equation~(\ref{eqn lambda conditions simp}), since, for all $\sigma,\sigma'\in\Sigma(X)$ with $\sigma\subseteq\sigma'$, one has
$\overline{f(\sigma)}=f(\sigmabar)$ and
\begin{equation*}
m(f(\sigmabar),f(\bar{\sigma'}))\cdot\lambda(\sigmabar)=\lambda(\bar{\sigma'})\cdot l(\sigmabar,\bar{\sigma'})\,.
\end{equation*}
Thus, $(\Sigma(f),\Sigma(\lambda))$ is a morphism in $\lSimp{}^A$.

Clearly, $\Sigma$ maps the identity morphism $(\id_X,1)$ of $(X,l)$ to the identity morphism of $(X,\Sigma(X),\Sigma(l))$. 
Finally, suppose that $(f,\lambda)\colon (X,l)\to(Y,m)$ and $(e,\mu)\colon (Y,m)\to(Z,n)$ are morphisms in $\lPoset{}^A$. 
We need to show that $\Sigma((e,\mu)\circ(f,\lambda))=\Sigma(e,\mu)\circ\Sigma(f,\lambda)$. 
But the left hand side equals $(e\circ f, \Sigma((\mu* f)\circ\lambda))$ and the right hand side equals $(e\circ f, (\Sigma(\mu)*f)\circ\Sigma(\lambda))$, and both of the second components are given by the function $\scrS\to A$, $\sigma\mapsto \mu(f(\sigmabar))\cdot\lambda(\sigmabar)$.

The functor $\Sigma\colon \Poset^A\to\Simp^A$ restricts to a functor $\poset^A\to\simp^A$.
\end{nothing}

\begin{lemma}\label{lem Sigma and homotopy}
If $(X,l)$ and $(Y,m)$ are $A$-monomial posets and $(f,\lambda),(f',\lambda')\colon (X,l)\to (Y,m)$ are homotopic in $\Poset^A$ then $\Sigma(f,\lambda)$ and $\Sigma(f',\lambda')$ are homotopic in $\Simp{}^A$. In particular, $\Sigma$ induces functors $\overline{\Sigma}\colon\scrH(\Poset^A)\to\scrH(\Simp^A)$ and $\overline{\Sigma}\colon\scrH(\poset^A)\to\scrH(\simp^A)$ such that the diagrams
\begin{diagram}[75]
   \movevertex(-10,0){\Poset^A} & \Ear[40]{\Sigma} & \movevertex(10,0){\Simp^A} & & & &
		\movevertex(-10,0){\lposet{}^A} & \Ear[40]{\Sigma} & \movevertex(10,0){\lsimp{}^A} &&
   \movearrow(-15,0){\sar} & & \movearrow(15,0){\sar} & & & &
		\movearrow(-15,0){\sar} & & \movearrow(15,0){\sar} &&
   \movevertex(-20,0){\scrH(\Poset^A)} & \Ear[40]{\overline{\Sigma}} & \movevertex(20,0){\scrH(\Simp^A)} & & & &
		\movevertex(-20,0){\scrH(\lposet{}^A)} & \Ear[40]{\overline{\Sigma}} & \movevertex(20,0){\scrH(\lsimp{}^A)} &&
\end{diagram}
commute, where the vertical arrows are the canonical functors.
\end{lemma}

\begin{proof}
We may assume that $(f,\lambda)\le (f', \lambda')$. 
Using the construction and notation in \ref{noth refining <=} together with (\ref{eqn refined morphisms}) and (\ref{eqn refined <=}), it suffices to show that $\Sigma(f_i,\lambda_i)-\Sigma(f_{i+1},\lambda_{i+1})$, for $i=0,\ldots,n$. So fix $i\in\{0,\ldots,n\}$. In order to prove $(f_i,\Sigma(\lambda_i))-(f_{i+1},\Sigma(\lambda_{i+1}))$, we need to show that (i) $f_i-f_{i+1}$ as morphism in $\Hom_{\Simp}((X,\Sigma(X)), (Y,\Sigma(Y)))$,  and (ii) $\iota_{f_i-f_{i+1},\Sigma(m)}\circ \Sigma(\lambda_i)=\Sigma(\lambda_{i+1})$.

\smallskip
(i) Let $\sigma=\{x_0,\ldots, x_k\}\in\Sigma(X)$ with $x_0<\cdots<x_k$. If $\sigma$ doesn't contain an element of height $n-i$ then $f_i=f_{i+1}$ by (\ref{eqn f i explicit}). And if $x_j\in\sigma$ has height $n-i$ then 
\begin{equation*}
f_i(\sigma)\cup f_{i+1}(\sigma) = \{f(x_0), \ldots, f(x_j), f'(x_j),\ldots,f'(x_k)\}\in \Sigma(Y)\,,
\end{equation*}
since $f(x_0)\le \cdots \le f(x_j)\le f'(x_j)\le \cdots \le f'(x_k)$.

\smallskip
(ii) Let $\sigma\in \Sigma(X)$ and set $x:=\sigmabar$. Since the maximal elements of $f_i(\sigma)$, $f_{i+1}(\sigma)$ and $f_i(\sigma)\cup f_{i+1}(\sigma)$ are $f_{i}(x)$, $f_{i+1}(x)$, and $f_{i+1}(x)$, respectively, by the definition of $\Sigma(m)$, $\Sigma(\lambda_i)$ and $\Sigma(\lambda_{i+1})$, we need to show that 
\begin{equation*}
m(f_i(x),f_{i+1}(x))\cdot \lambda_i(x)  = m(f_{i+1}(x),f_{i+1}(x))\cdot \lambda_{i+1}(x)\,.
\end{equation*}
But this follows immediately from $m(f_{i+1}(x), f_{i+1}(x))=1$ and $(f_i,\lambda_i)\le (f_{i+1},\lambda_{i+1})$ in $\Hom_{\Poset^A}((X,l),(Y,m))$.	
\end{proof}

\begin{lemma}
{\rm (a)} The functor $\Sigma \colon \Poset^A \to \Simp^A$ is faithful.

\smallskip
{\rm (b)} The functor $\overline{\Sigma} \colon \scrH(\Poset^A) \to \scrH(\Simp^A)$ is faithful.
\end{lemma}

\begin{proof}
(a) Let $(f,\lambda),(f',\lambda')\colon (X,l) \to (Y,m)$ be in $\Poset^A$ with $\Sigma(f,\lambda)= \Sigma(f',\lambda')$. Then, for every $x \in X$, we have
$$f(x)= \Sigma(f)(x)= \Sigma(f')(x)= f'(x)$$
and
$$\lambda(x)= \Sigma(\lambda)(\{x\})= \Sigma(\lambda')(\{x\})= \lambda'(x).$$
Thus, $(f,\lambda)=(f',\lambda')$.

\smallskip
(b) Let $(X,l)$ and $(Y,m)$ be $A$-monomial posets. 
Let $[f,\lambda], [f',\lambda']\colon (X,l) \to (Y,m)$ be morphisms in $\scrH(\lposet{}^A)$ such that 
$\overline{\Sigma}([f,\lambda])-\overline{\Sigma}([f',\lambda'])$. We want to show that $(f,\lambda)\sim(f',\lambda')$. Note that  since 
$\overline{\Sigma}([f,\lambda])-\overline{\Sigma}([f',\lambda'])$, for any $x \in X$ we have $\overline{\Sigma(f)}(x) \cup \overline{\Sigma(f')}(x)= f(x)\cup f'(x) \in \Sigma(Y)$. So either $f(x) \leq f'(x)$ or $f'(x)< f(x)$. Thus, we can define a map  
$(f\cup f', \lambda\cup\lambda')\colon (X,l) \to (Y,m)$ in $\lPoset{}^A$ such that 
\begin{equation*}
(f\cup f')(x):= \begin{cases} f'(x) & \text{if $f(x) \leq f'(x)$,} \\
f(x) & \text{if $f'(x)< f(x)$}
\end{cases}
\end{equation*}
and
\begin{equation*}
(\lambda\cup \lambda')(x):= \begin{cases} \lambda'(x) & \text{if $f(x) \leq f'(x)$,} \\
\lambda(x) & \text{if $f'(x)< f(x)$}
\end{cases}
\end{equation*}
for any $x \in X$. Now we show that $(f\cup f', \lambda\cup\lambda')$ is in $\lPoset{}^A$. Since $f,f' \in \lPoset{}$, we have $f\cup f' \in \lPoset{}$.

%Given $x \leq y$ in $X$, we need to show that $(f\cup f')(x) \leq (f\cup f')(y)$. We have four cases to consider.
%\smallskip(i) If $f(x) \leq f'(x)$ and $f(y) \leq f'(y)$ then we get
%$$(f\cup f')(x) = f'(x) \leq f'(y)= (f\cup f')(y).$$

%\smallskip(ii) If $f'(x) < f(x)$ and $f'(y) < f(y)$ then
%$$(f\cup f')(x)= f(x) \leq f(y)= (f\cup f')(y).$$

%\smallskip(iii) If $f'(x)< f(x)$ and $f(y) \leq f'(y)$ then
%$$(f\cup f')(x)= f(x) \leq f(y) \leq f'(y)= (f\cup f')(y).$$

%\smallskip(iv) If $f(x) \leq f'(x)$ and $f'(y) < f(y)$ then
%$$(f\cup f')(x) = f'(x) \leq f'(y) < f(y)= (f\cup f')(y).$$

Now we show that $\lambda \cup \lambda'\colon l \to m \circ f\cup f'$ is a natural transformation. Given $x \le y$ in $X$, we need to show that 
$$(\lambda\cup\lambda')(y)l(x,y)=m((f\cup f')(x), (f\cup f')(y))(\lambda\cup\lambda')(x).$$
We have four cases to consider.

\smallskip
(i) If $f(x) \le f'(x)$ and $f(y) \le f'(y)$ then
\begin{align*}
(\lambda\cup\lambda')(y)l(x,y) &=  \lambda'(y)l(x,y)=m(f'(x), f'(y))\lambda'(x)\\
&= m((f \cup f')(x), (f\cup f')(y))(\lambda\cup\lambda')(x).
\end{align*}

\smallskip(ii) The case $f'(x)< f(x)$ and $f'(y)< f(y)$ is similar to part (i).

%If $f'(x)< f(x)$ and $f'(y)< f(y)$ then

%$$(\lambda\cup\lambda')(y)l(x,y)= \lambda(y)l(x, y)= m(f(x), f(y))\lambda(x)$$
%$$=m((f\cup f')(x), (f\cup f')(y))(\lambda\cup\lambda')(x).$$

\smallskip(iii) If $f(x)\le f'(x)$ and $f'(y) < f(y)$ then
\begin{align*}
(\lambda\cup\lambda')(y)l(x,y) &=  \lambda(y)l(x,y)=m(f(x), f(y))\lambda(x)\\
&=  m(f'(x), f(y))m(f(x), f'(x))\lambda(x)\\
&=  m(f'(x), f(y))\Sigma(m)(f(x), f(x)\cup f'(x))\Sigma(\lambda)(x)\\
&=  m(f'(x), f(y))\Sigma(m)( f'(x), f(x)\cup f'(x))\Sigma(\lambda')(x)\\
%$$= m(f'(x), f(y))m(f'(x), f'(x))\lambda'(x)$$
&=  m(f'(x), f(y))\lambda'(x)=m((f\cup f')(x), (f\cup f')(y))(\lambda\cup\lambda')(x)\,.\\
\end{align*}

\smallskip(iv) The case $f'(x)< f(x)$ and $f'(y)< f(y)$ is similar to part (iii).

%If $f'(x)< f(x)$ and $f(y) \le f'(y)$ then
%$$(\lambda\cup\lambda')(y)l(x,y)= \lambda'(y)l(x,y)= m(f'(x), f'(y))\lambda'(x)$$
%$$=m(f(x), f'(y))m(f'(x), f(x))\lambda'(x)$$
%$$=m(f(x), f'(y))\Sigma(m)(f'(x), f(x)\cup f'(x))\Sigma(\lambda')(x)$$
%$$= m(f(x), f'(y))\Sigma(m)(f(x), f(x)\cup f'(x))\Sigma(\lambda)(x)$$
%$$= m(f(x), f'(y))m(f(x), f(x))\lambda(x)=m(f(x), f'(y))\lambda(x)$$
%$$= m((f\cup f')(x), (f\cup f')(y))(\lambda\cup\lambda')(x).$$

\smallskip
Thus, $(f\cup f', \lambda\cup\lambda')\colon (X,l) \to (Y,m)$ is a morphism in $\lPoset{}^A$. Next we show that $(f,\lambda) \le (f\cup f', \lambda\cup\lambda')$.

Let $x \in X$. We have two cases to consider $f(x)\le f'(x)$ and $f'(x)< f(x)$. If $f(x) \le f'(x)$ then we obviously have $f \le f\cup f'$ and moreover
\begin{align*}
\iota_{f- f\cup f', m}(x) \circ \lambda(x) &= m(f(x), (f\cup f')(x))\lambda(x)= m(f(x), f'(x))\lambda(x)\\
&=\Sigma(m)(f(x), f(x)\cup f'(x))\Sigma(\lambda)(x)= \Sigma(\lambda')(x)\\
&=\lambda'(x)= (\lambda\cup\lambda')(x)\,.
\end{align*}
This means that $(f,\lambda) \le (f\cup f', \lambda\cup\lambda')$ in this case. The same result follows similarly in the case $f'(x)< f(x)$.
%If $f'(x)< f(x)$ then we obviously have
%$f' \le f\cup f'$. Moreover, we have 
%$$\iota_{f'- f\cup f', m}(x) \circ \lambda(x) = m(f'(x), (f\cup f')(x))\lambda(x)= m(f'(x),f(x))\lambda(x)$$
%$$=\Sigma(m)(f'(x), f(x)\cup f'(x))\Sigma(\lambda)(x)= \Sigma(\lambda')(x)=\lambda'(x)= (\lambda\cup\lambda')(x).$$
The same way one can prove $(f', \lambda') \le (f\cup f', \lambda\cup\lambda')$.  Altogether we obtain $(f,\lambda)\sim(f',\lambda')$ and the proof is complete.
\end{proof}

As a consequence o the above Lemma, also the functors $\Sigma \colon \poset^A \to \simp^A$ and $\overline{\Sigma} \colon \scrH(\poset^A) \to \scrH(\simp^A)$ are faithful.

\begin{nothing}{\em The functor $\Pi\colon \lSimp{}^A\to\lPoset{}^A$.}\quad For any $A$-monomial simplicial complex $(S,\scrS,l)$ we define $\Pi(S,\scrS,l):=(\scrS,l)$, an $A$-monomial poset, where $\scrS$ is viewed as partially ordered set via inclusion of subsets of $S$. 
For any morphism $(f,\lambda)\colon (S,\scrS,l)\to(T,\scrT,m)$ in $\lSimp{}^A$ we define $\Pi(f,\lambda):=(f,\lambda)$, where $f$ on the right hand side denotes the map $f\colon \scrS\to\scrT$, $\sigma\mapsto f(\sigma),$ induced by $f\colon S\to T$. 
It is straightforward to check that this defines a functor $\Pi\colon\lSimp{}^A\to\lPoset{}^A$ which restricts to a functor $\Sigma\colon \lsimp{}^A\to\lposet{}^A$.
\end{nothing}

\begin{lemma}
	If $(f,\lambda),(f',\lambda')\colon (S,\scrS,l)\to(T,\scrT,m)$ are homotopic in $\lSimp{}^A$ then $\Pi(f,\lambda)$ and $\Pi(f',\lambda')$ are homotopic in $\lPoset{}^A$. In particular, $\Pi$ induces a functor $\overline{\Pi}\colon\scrH(\lSimp{}^A)\to\scrH(\lPoset{}^A)$ which restricts to a functor  $\overline{\Pi}\colon\scrH(\lsimp{}^A)\to\scrH(\lposet{}^A)$ such that the diagrams
	\begin{diagram}[75]
		\movevertex(-10,0){\lSimp{}^A} & \Ear[40]{\Pi} & \movevertex(10,0){\lPoset{}^A} & & & & 
		\movevertex(-10,0){\lsimp{}^A} & \Ear[40]{\Pi} & \movevertex(10,0){\lposet{}^A} &&
		\movearrow(-15,0){\sar} & & \movearrow(15,0){\sar} & & & & 
		\movearrow(-15,0){\sar} & & \movearrow(15,0){\sar} &&
		\movevertex(-20,0){\scrH(\lSimp{}^A)} & \Ear[40]{\overline{\Pi}} & \movevertex(20,0){\scrH(\lPoset{}^A)} & & & & 
		\movevertex(-20,0){\scrH(\lsimp{}^A)} & \Ear[40]{\overline{\Pi}} & \movevertex(20,0){\scrH(\lposet{}^A)} &&
	\end{diagram}
	commute, where the vertical arrows are the canonical functors.
\end{lemma}

\begin{proof}
We may assume that $(f,\lambda)-(f', \lambda')$ in $\lSimp{}^A$. We define the pair $(f\cup f', \eta_{\lambda,\lambda'})\colon \Pi(S,\scrS,l) \to \Pi(T,\scrT,m)$ where 
$f\cup f'\colon \scrS \to \scrT$ is given by $(f\cup f')(\sigma)=f(\sigma)\cup f'(\sigma)$  and 
$\eta_{\lambda,\lambda'}\colon l \to m\circ f \cup f'$ is given by 
$$\eta_{\lambda,\lambda'}(\sigma)= m(f'(\sigma), f(\sigma)\cup f'(\sigma))\cdot\lambda'(\sigma)$$ for any $\sigma \in \scrS$. We want to show that $(f\cup f', \eta_{\lambda,\lambda'})$ is in $\lposet{}^A$ and $(f,\lambda)-(f\cup f', \eta_{\lambda,\lambda'})-(f',\lambda')$ in $\lposet{}^A$. Obviuosly, $f\cup f'\colon \scrS \to \scrT$ is a map of posets. For any  $\sigma, \sigma' \in \scrS$ with $\sigma \subseteq \sigma'$, we have 
\begin{align*}
\eta_{\lambda,\lambda'}(\sigma')\cdot l(\sigma, \sigma')&=m(f'(\sigma'), f(\sigma')\cup f'(\sigma'))\cdot \lambda'(\sigma')\cdot  l(\sigma, \sigma')\\
&=m(f'(\sigma'), f(\sigma')\cup f'(\sigma'))\cdot m(f'(\sigma), f'(\sigma'))\cdot \lambda'(\sigma)\\
%&=m(f'(\sigma), f(\sigma')\cup f'(\sigma'))\cdot \lambda'(\sigma)\\
&=m(f(\sigma)\cup f'(\sigma), f(\sigma')\cup f'(\sigma'))\cdot m(f'(\sigma), f(\sigma)\cup f'(\sigma))\cdot \lambda'(\sigma)\\
&= m(f(\sigma)\cup f'(\sigma), f(\sigma')\cup f'(\sigma'))\cdot \eta_{\lambda,\lambda'}(\sigma).
\end{align*}
and this implies that $\eta_{\lambda,\lambda'}\colon l \to m\circ f \cup f'$ is a natural transformation. Now we will show that (i) $(f,\lambda)\le (f\cup f', \eta_{\lambda,\lambda'})$  and (ii) $(f\cup f', \eta_{\lambda,\lambda'})\ge(f',\lambda')$.

\smallskip
(i) Obviously, $f \le f\cup f'$. We get $\iota_{f\le f\cup f', m}\circ\lambda= \eta_{\lambda,\lambda'}$
because from  Equation (\ref{eqn - condition in simp}) we have
\begin{align*}
\iota_{f- f\cup f', m}(\sigma)\cdot \lambda(\sigma) &= m(f(\sigma), f(\sigma)\cup f'(\sigma))\cdot \lambda(\sigma)\\
&= m(f'(\sigma), f(\sigma)\cup f'(\sigma))\cdot \lambda'(\sigma) = \eta_{\lambda,\lambda'}(\sigma)
\end{align*}
for any $\sigma \in \scrS$.
So $(f, \lambda) \le (f\cup f', \eta_{\lambda,\lambda'})$.

\smallskip
(ii) The case $f' \le f\cup f'$ follows similarly from the definition of $\iota_{f'- f\cup f', m}$.

%Obviously, $f' \le f\cup f'$.  We get $\iota_{f'- f\cup f', m}\circ \lambda'= \eta_{\lambda,\lambda'}$
%because we have
%$$\iota_{f'- f\cup f', m}(\sigma)\cdot \lambda'(\sigma)=m(f'(\sigma), f(\sigma)\cup f'(\sigma))\cdot \lambda'(\sigma)=\eta_{\lambda,\lambda'}(\sigma)$$
%for any $\sigma \in \scrS$. So $(f', \lambda') \le (f\cup f', \eta_{\lambda,\lambda'})$.

\smallskip
Altogether we get $(f,\lambda)-(f\cup f', \eta_{\lambda,\lambda'})-(f',\lambda')$ and so $(f,\lambda) \sim (f',\lambda')$.
\end{proof}

\begin{lemma}
{\rm (a)} The functor $\Pi\colon \lSimp{}^A \to \lPoset{}^A$ is faithful.

\smallskip
{\rm (b)} The functor $\overline{\Pi} \colon \scrH(\lSimp{}^A) \to \scrH(\lPoset{}^A)$ is faithful.
\end{lemma}

\begin{proof}
Part~(a) is obvious by the definition of the functor $\Pi$.
To prove Part~(b), let $[f,\lambda], [f',\lambda'] \colon (S,\scrS,l) \to (T,\scrT,m)$ be morphisms in $\scrH(\lposet{}^A)$ such that $\overline{\Pi}([f,\lambda])= \overline{\Pi}([f,\lambda])$. We want to show that $(f,l) \sim(f',l')$ in $\lSimp{}^A$. We may assume that $\Pi(f,\lambda) \leq \Pi(f',\lambda')$. Then for any $\sigma \in \scrS$, we have $f(\sigma) \subseteq f'(\sigma)$ and so $f(\sigma) \cup f'(\sigma)= f'(\sigma) \in\scrS$. Since $\Pi(f,\lambda) \leq \Pi(f',\lambda')$, we have $\iota_{f-f',m} \circ \lambda = \lambda'$. Thus, we get $(f,\lambda) - (f,\lambda)$ in $\lSimp{}^A$.
\end{proof}

\begin{nothing}{\em The functor $| \cdot | \colon \lsimp{}^A \to \ltop{}^A$.}
We recall the functor $| \cdot | \colon \lsimp{} \to \ltop{}$ following \cite[Chapter~3]{Spanier}. Let $(S,\scrS)$ be a simplicial complex. We set
$$| S, \scrS | = \{\alpha \in  \RR^S \mid \alpha(s) \geq 0, \supp(\alpha) \in \scrS\,\,\text{and}\,\, \sum\limits_{s \in S}\alpha(s)=1 \}$$
where $\supp(\alpha) = \{s \in S \mid \alpha(s) \neq 0 \}$ and $\RR^S$ is the set of functions from $S$ to $\RR$. Now consider the coproduct $ \coprod\limits_{\sigma \in \scrS} | \sigma |$ where 
$$| \sigma |= \{\alpha \in \RR^\sigma \mid \sum\limits_{s \in \sigma}\alpha(s)=1, \alpha(s) \geq 0\,\, \forall s \in \sigma \}.$$
We let $(\sigma, \alpha)$ denote the point $\alpha \in |\sigma|$  in $ \coprod\limits_{\sigma \in \scrS} | \sigma |$. In particular, $|S,\scrS|=   \faktor{\coprod\limits_{\sigma \in \scrS}|\sigma |}{\sim}$
where the relation $\sim$ is defined by

\begin{equation*}
(\sigma,\alpha)\sim (\tau,\beta) :\iff \alpha =\beta \quad\text{in}\quad \RR^S
\end{equation*}
for any $(\sigma, \alpha), (\tau,\beta) \in \coprod\limits_{\sigma \in \scrS}|\sigma|$. 

Given a morphism $f \colon (S,\scrS) \to (T,\scrT)$ in $\lsimp{}^A$, we define a morphism $|f|\colon |S,\scrS| \to |T,\scrT|$ such that
$|f|(\alpha)(t):= \sum\limits_{s \in f^{-1}(t)}\alpha(s)$ for any $\alpha \in |S,\scrS|$ and $t \in T$.

Now we will define the functor $|\cdot|\colon \lsimp{}^A \to \ltop{}^A$. The following construction is similar to the  construction 3.1 in \cite{Bal18}. For any $A$-monomial simplicial complex $(S,\scrS,l),$ we set $|S,\scrS,l|:= \faktor{E_{(S,\scrS,l)}}{\sim}$ where 
$$E_{(S,\scrS,l)}:= \coprod\limits_{\sigma \in \scrS}|\sigma |\times A=\{(\sigma, \alpha, a) \mid \sigma \in \scrS, \alpha \in |\sigma|, a\in A \}$$
and the relation $\sim$ is given by 
\begin{equation*}
(\sigma,\alpha,a)\sim (\tau,\beta,b) :\iff \alpha =\beta \quad\text{in}\quad \RR^S \quad\text{and}\quad al(\supp(\alpha),\sigma)= bl(\supp(\beta), \tau)
\end{equation*}
for any $(\sigma,\alpha,a), (\tau,\beta,b) \in E_{(S,\scrS,l)}$. It's straightforward to show that $\sim$ is an equivalence relation on $E_{(S,\scrS,l)}$. We denote the equivalence class containing $(\sigma, \alpha, a)$ by $[\sigma, \alpha, a]$. Obviously, the set $|S,\scrS,l|$ inherits the quotient topology and it has a continuous $A$-action given by
\begin{equation*}
\forall b \in A, [\sigma, \alpha, a] \in |S,\scrS,l|,\quad b\cdot[\sigma,\alpha,a]= [\sigma, \alpha, ba].
\end{equation*}
We define a map $\pi \colon |S,\scrS,l| \to |S,\scrS|$ given by $\pi([\sigma,\alpha,a])=\alpha$ for any $[\sigma,\alpha,a] \in |S,\scrS,l|$. Obviously, the map $\pi$ is surjective and continuous. Now we show that $\pi\colon |S,\scrS,l| \to |S,\scrS|$ is an $A$-fibered bundle. Let $\alpha \in |S,\scrS|$. We set
$$\varepsilon := \min\{ \alpha(s) \mid s \in \supp(\alpha) \}$$
and 
$$U_\alpha := \{\beta \in |S,\scrS| \mid |\beta(s)-\alpha(s)| < \varepsilon \quad \forall s \in S  \}.$$
Then $U_\alpha$ is open in $|S,\scrS|$. Note that for any $\beta \in U_\alpha$ we have $\supp(\alpha) \subseteq \supp(\beta)$ and 
$$\pi^{-1}(U_\alpha)= \{[\sigma, \beta, a] \mid \supp(\alpha)\subseteq \sigma, |\beta -\alpha| < \varepsilon, a \in A  \}.$$ 
The map $\varphi \colon \pi^{-1}(U_\alpha) \to U_\alpha \times A$, $[\sigma, \beta,a] \to (\beta, al(\supp(\alpha), \sigma)^{-1})$ is an $A$-equivariant homeomorphism such that
$$\xymatrix@R=4ex@C=3ex{
	\pi^{-1}(U_\alpha)\ar[rr]^-{\varphi}\ar[rd]_-{\pi}&&U_\alpha\times A\ar[ld]^{p_1}\\
	&U_\alpha&
}
$$
commutes. Thus, $\pi \colon |S,\scrS,l|\to |S,\scrS|$ is an $A$-fibered bundle. Now let $(f,\lambda) \colon (S,\scrS,l) \to (T,\scrT,m)$ be a morphism in $\lsimp{}^A$. We will define a morphism $|f,\lambda|\colon |S,\scrS,l| \to |T,\scrT,m|$ in $\ltop{}^A$ such that
$$\xymatrix@C=7ex@R=7ex{
	|S, \scrS, l|\ar[r]^-{|f, \lambda|}\ar[d]_-{\pi}&|T,\scrT,m|\ar[d]^-{\pi}\\
	|S,\scrS|\ar[r]_-{|f|}& |T,\scrT|}$$
commutes. Consider the restriction of the continuous function $\scrS\times\RR^S\times A \to \scrT\times\RR^T\times A$, $(\sigma,\alpha, a) \mapsto (f(\sigma), |f|(\alpha), a\lambda(\sigma))$ to $E_{(S,\scrS,l)} \to E_{(T,\scrT,m)}$. Moreover, it induces a continuous map on equivalence classes $|f,\lambda| \colon |S,\scrS,l| \to |T,\scrT,m|$. Thus, $|f,\lambda| \colon |S,\scrS,l| \to |T,\scrT,m|$ is a morphism in $\ltop{}^A$.
\end{nothing}
\begin{lemma}
If $(S,\scrS,l)$ and $(T, \scrT, m)$ are finite $A$-monomial simplicial complexes and $(f,\lambda),(f',\lambda')\colon(S,\scrS,l)\to (T, \scrT, m)$ are homotopic then $|(f,\lambda)|$ and $|(f',\lambda')|$ are homotopic in $\ltop{}^A$. In particular, $|\cdot|$ induces a functor $\overline{|\cdot|}\colon\scrH(\lsimp{}^A)\to\scrH(\ltop{}^A)$ such that the diagram
\begin{diagram}[75]
	\movevertex(-10,0){\lsimp{}^A} & \Ear[40]{|\cdot|} & \movevertex(10,0){\ltop{}^A} &&
	\movearrow(-15,0){\sar} & & \movearrow(15,0){\sar} &&
	\movevertex(-20,0){\scrH(\lsimp{}^A)} & \Ear[40]{\overline{|\cdot|}} & \movevertex(20,0){\scrH(\ltop{}^A)} &&
\end{diagram}
commutes, where the vertical arrows are the canonical functors.	
\end{lemma}
\begin{proof}
We may assume that $(f,\lambda)-(f',\lambda')$. We define the pair $(H,h)\colon (\pi \colon |S,\scrS, l| \to |S,\scrS|)\otimes(1_I\colon A\times I \to I)\to (\pi'\colon |T,\scrT, m| \to |T,\scrT|)$ where $h(\alpha,x)= (1-x)|f|(\alpha)+x|f'|(\alpha)$ 
for any $\alpha \in |S,\scrS|$, $x \in I$ and
$$H([\sigma,\alpha,a],x)= [f(\sigma)\cup f'(\sigma), h(\alpha,x), a\lambda(\sigma)^{-1}m(f(\sigma),f(\sigma)\cup f'(\sigma))^{-1} ]$$
for any $[\sigma,\alpha,x] \in |S,\scrS,l|$ and $x \in I$. First we show that $(H,h)$ is a morphism in $\ltop{}^A$. Obviously, we have $\sum\limits_{t \in T} h(\alpha,x)(t)=1$ and $h(\alpha,x)(t) \geq 0$ for any $t \in T$. Given $(\alpha,x) \in |S,\scrS|\times I$, we have
\begin{equation*}
\supp(h(\alpha,x))= \begin{cases} f(\supp(\alpha)) & \text{if $x=0$,} \\
f'(\supp(\alpha)) & \text{if $x=1$.}\\
f(\supp(\alpha))\cup f'(\supp(\alpha)) & \text{otherwise.}
\end{cases}
\end{equation*}

Since the maps $f,f'$ are simplicial maps and $f-f'$, we get $\supp(h(\alpha,x)) \in \scrT$.  So $h(\alpha,x) \in |T,\scrT|$. Moreover, we have $h(\alpha,0)= |f|(\alpha)$ and $h(\alpha,1)= |f'|(\alpha)$ for every $\alpha \in |S,\scrS|$.  Now we work on $H \colon |S,\scrS,l|\times I \to |T,\scrT,m|$, $([\sigma,\alpha,a],x) \mapsto [f(\sigma)\cup f'(\sigma), h(\alpha,x), a\lambda(\sigma)^{-1}m(f(\sigma),f(\sigma)\cup f'(\sigma))^{-1}]$. Since $f-f'$, we have $f(\sigma)\cup f'(\sigma) \in \scrT$. We get $h(\alpha,x) \in |f(\sigma) \cup f'(\sigma)|$ because $\alpha  \in |\sigma|$. We need to show that $H([\sigma,\alpha,a],x)= H([\tau,\beta,b],x)$ for any $x \in I$ and for any $(\sigma,\alpha, a) , (\tau, \beta, b) \in |S,\scrS,l|$ such that $(\sigma,\alpha, a)\sim(\tau, \beta, b)$. Note that $(\sigma,\alpha, a)\sim(\tau, \beta, b)$ gives $h(\alpha,x)=h(\beta,x)$. We need to show that
$$a\lambda(\sigma)^{-1}m(f(\sigma), f(\sigma)\cup f'(\sigma))^{-1}m(\supp(h(\alpha,x)), f(\sigma)\cup f'(\sigma))$$
$$= b\lambda(\tau)^{-1} m(f(\tau), f(\tau)\cup f'(\tau))^{-1}m(\supp(h(\beta,x)), f(\tau)\cup f'(\tau) )$$

We have three cases to consider. (i) $x=0$, (ii) $x=1$ and (iii) $x \neq 0,1$.

\smallskip
(i) Follows from
%$$a\lambda(\sigma)^{-1}m(f(\sigma), f(\sigma)\cup f'(\sigma))^{-1}m(\supp(h(\alpha,0)), f(\sigma)\cup f'(\sigma))$$ 
$$= a\lambda(\sigma)^{-1}m(f(\sigma), f(\sigma)\cup f'(\sigma))^{-1}m(f(\supp(
\alpha)), f(\sigma)\cup f'(\sigma))$$
%$$= a\lambda(\sigma)^{-1}m(f(\sigma), f(\sigma)\cup f'(\sigma))^{-1}m(f(\sigma), f(\sigma)\cup f'(\sigma))m(f(\supp(\alpha)), f(\sigma))$$
$$= a\lambda(\sigma)^{-1}m(f(\supp(\alpha)), f(\sigma))= al(\supp(\alpha), \sigma)\lambda(\supp(\alpha))^{-1}$$
$$= bl(\supp(\beta),\tau)\lambda(\supp(\beta))^{-1}=b\lambda(\tau)^{-1}m(f(\supp(\beta)), f(\tau))$$
%$$= b \lambda(\tau)^{-1}m(f(\tau), f(\tau)\cup f'(\tau))^{-1} m(f(\tau), f(\tau)\cup f'(\tau)) m(f(\supp(\beta)), f(\tau))$$
$$= b\lambda(\tau)^{-1} m(f(\tau), f(\tau)\cup f'(\tau))^{-1}m(f(\supp(\beta)), f(\tau)\cup f'(\tau) )$$
%$$= b\lambda(\tau)^{-1} m(f(\tau), f(\tau)\cup f'(\tau))^{-1}m(\supp(h(\beta,0)), f(\tau)\cup f'(\tau) )$$

\smallskip
(ii) Similar to part (i).
%Follows from
%$$a\lambda(\sigma)^{-1}m(f(\sigma), f(\sigma)\cup f'(\sigma))^{-1}m(\supp(h(\alpha,1)), f(\sigma)\cup f'(\sigma))$$
%$$=a\lambda'(\sigma)^{-1}m(f'(\sigma), f(\sigma)\cup f'(\sigma))^{-1}m(f'(\supp(\alpha)), f(\sigma)\cup f'(\sigma))$$
%$$=a\lambda'(\sigma)^{-1}m(f'(\sigma), f(\sigma)\cup f'(\sigma))^{-1}m(f'(\sigma), f(\sigma)\cup f'(\sigma))m(f'(\supp(\alpha)), f'(\sigma))$$
%$$=a\lambda'(\sigma)^{-1}m(f'(\supp(\alpha)), f'(\sigma))= al(\supp(\alpha), \sigma)\lambda'(\supp(\alpha))^{-1} $$
%$$=bl(\supp(\beta), \tau)\lambda'(\supp(\beta))^{-1}=b\lambda'(\tau)^{-1}m(f'(\supp(\beta)), f'(\tau))$$
%$$= b \lambda'(\tau)^{-1}m(f'(\tau), f(\tau)\cup f'(\tau))^{-1}m(f'(\tau), f(\tau)\cup f'(\tau) )m(f'(\supp(\beta)), f'(\tau))$$
%$$= b \lambda(\tau)^{-1}m(f(\tau), f(\tau)\cup f'(\tau))^{-1}m(f'(\supp(\beta)), f(\tau)\cup f'(\tau))$$
%$$= b\lambda(\tau)^{-1} m(f(\tau), f(\tau)\cup f'(\tau))^{-1}m(\supp(h(\beta,1)), f(\tau)\cup f'(\tau) )$$

\smallskip
(iii) Follows from
%$$a\lambda(\sigma)^{-1}m(f(\sigma), f(\sigma)\cup f'(\sigma))^{-1}m(\supp(h(\alpha,x)), f(\sigma)\cup f'(\sigma))$$
$$= a \lambda(\sigma)^{-1}m(f(\sigma), f(\sigma)\cup f'(\sigma))^{-1}m(f(\supp(\alpha))\cup f'(\supp(\alpha)) , f(\sigma) \cup f'(\sigma))$$
%$$= a \lambda(\sigma)^{-1}m(f(\sigma), f(\sigma)\cup f'(\sigma))^{-1}m(f(\supp(\alpha))\cup f'(\supp(\alpha)) , f(\sigma) \cup f'(\sigma))$$
%$$m(f(\supp(\alpha)), f(\supp(\alpha))\cup f'(\supp(\alpha)))m(f(\supp(\alpha)), f(\supp(\alpha))\cup f'(\supp(\alpha)))^{-1}$$
$$= a \lambda(\sigma)^{-1}m(f(\sigma), f(\sigma)\cup f'(\sigma))^{-1}m(f(\supp(\alpha)), f(\sigma) \cup f'(\sigma))$$
$$m(f(\supp(\alpha)), f(\supp(\alpha))\cup f'(\supp(\alpha)))^{-1}$$
%$$= a \lambda(\sigma)^{-1}m(f(\sigma), f(\sigma)\cup f'(\sigma))^{-1}m(f(\sigma), f(\sigma)\cup f'(\sigma))$$
%$$m(f(\supp(\alpha)), f(\sigma))m(f(\supp(\alpha)), f(\supp(\alpha))\cup f'(\supp(\alpha)))^{-1}$$
$$=a \lambda(\sigma)^{-1}m(f(\supp(\alpha)), f(\sigma))m(f(\supp(\alpha)), f(\supp(\alpha))\cup f'(\supp(\alpha)))^{-1}$$
$$= al(\supp(\alpha), \sigma)\lambda(\supp(\alpha))^{-1}m(f(\supp(\alpha)), f(\supp(\alpha)) \cup f'(\supp(\alpha)))$$
$$=b l(\supp(\beta), \tau)\lambda(\supp(\beta))^{-1} m(f(\supp(\beta)), f(\supp(\beta)) \cup f'(\supp(\beta)))^{-1}$$
$$=b\lambda(\tau)^{-1}m(f(\supp(\beta)), f(\tau))m(f(\supp(\beta)), f(\supp(\beta)) \cup f'(\supp(\beta)))^{-1}$$
%$$=b\lambda(\tau)^{-1} m(f(\tau), f(\tau)\cup f'(\tau))^{-1}m(f(\tau), f(\tau)\cup f'(\tau))$$
%$$m(f(\supp(\beta)), f(\tau))m(f(\supp(\beta)), f(\supp(\beta)) \cup f'(\supp(\beta)))^{-1}$$
$$=b\lambda(\tau)^{-1} m(f(\tau), f(\tau)\cup f'(\tau))^{-1}$$
$$m (f(\supp(\beta)), f(\tau)\cup f'(\tau))m(f(\supp(\beta)), f(\supp(\beta)) \cup f'(\supp(\beta)))^{-1}$$
%$$=b\lambda(\tau)^{-1} m(f(\tau), f(\tau)\cup f'(\tau))^{-1}m (f(\supp(\beta)) \cup f'(\supp(\beta)), f(\tau)\cup f'(\tau)) $$
%$$m(f(\supp(\beta)), f(\supp(\beta)) \cup f'(\supp(\beta)))m(f(\supp(\beta)), f(\supp(\beta)) \cup f'(\supp(\beta)))^{-1}$$
$$= b\lambda(\tau)^{-1} m(f(\tau), f(\tau)\cup f'(\tau))^{-1}m (f(\supp(\beta)) \cup f'(\supp(\beta)), f(\tau)\cup f'(\tau)) $$
%$$= b \lambda(\tau)^{-1}m(f(\tau), f(\tau)\cup f'(\tau))^{-1}m(\supp(h(\beta,x)), f(\tau)\cup f'(\tau))^{-1}.$$
So we have $H([\sigma,\alpha,a],x)= H([\tau,\beta,b],x)$ and $H$ is well-defined. It's not hard to show that the diagram
$$\xymatrix@C=7ex@R=7ex{
	|S,\scrS,l|\times I\ar[r]^-{H}\ar[d]_-{\pi\times 1_I}&|T,\scrT,m|\ar[d]^-{\pi'}\\
	|S,\scrS|\times I\ar[r]_-{h}& |T,\scrT|                              .
}$$
is commutative so $(H,h)$ is a morphism in $\ltop{}^A$. 
Now we show that $H([\sigma,\alpha,a],0)=|f,\lambda|([\sigma,\alpha,a])$. Since 
$$a\lambda(\sigma)^{-1}m(f(\sigma), f(\sigma)\cup f'(\sigma) )^{-1}m(f(\supp(\alpha) ), f(\sigma)\cup f'(\sigma) )$$
%$$= a\lambda(\sigma)^{-1}m(f(\sigma), f(\sigma)\cup f'(\sigma))^{-1}m(f(\sigma), f(\sigma)\cup f'(\sigma) ) m(f(\supp(\alpha)), f(\sigma))$$
$$=a\lambda(\sigma)^{-1}m(f(\supp(\alpha)), f(\sigma)),$$

we have 
$$ H([\sigma, \alpha, a], 0)= [f(\sigma)\cup f'(\sigma), h(\alpha,0), a \lambda(\sigma)^{-1}m(f(\sigma), f(\sigma)\cup f'(\sigma))^{-1}]$$
$$= [f(\sigma)\cup f'(\sigma), |f|(\alpha), a \lambda(\sigma)^{-1}m(f(\sigma), f(\sigma)\cup f'(\sigma))^{-1}] $$
$$= [f(\sigma),  |f|(\alpha), a\lambda(\sigma)^{-1}]= |f,\lambda|([\sigma, \alpha,a]).$$

The proof of $H([\sigma,\alpha,a],1)=|f',\lambda'|([\sigma,\alpha,a])$ for any $[\sigma,\alpha,a] \in |S,\scrS,l|$ follows similarly.
%Since
%$$a \lambda(\sigma)^{-1}m(f(\sigma), f(\sigma)\cup f'(\sigma) )^{-1}m(f'(\supp(\alpha) ), f(\sigma)\cup f'(\sigma))$$
%$$= a\lambda'(\sigma)^{-1}m(f'(\sigma), f(\sigma)\cup f'(\sigma) )^{-1}m(f'(\sigma), f(\sigma)\cup f'(\sigma))m(f'(\supp(\alpha)), f'(\sigma))$$
%$$= a\lambda'(\sigma)^{-1}m(f'(\supp(\alpha)), f'(\sigma)),$$
%we have 
%$$H([\sigma, \alpha, a], 1)=[f(\sigma)\cup f'(\sigma), h(\alpha,1), a \lambda(\sigma)^{-1}m(f(\sigma), f(\sigma)\cup f'(\sigma))^{-1}]$$
%$$= [f(\sigma)\cup f'(\sigma), |f'|(\alpha), a \lambda(\sigma)^{-1}m(f(\sigma), f(\sigma)\cup f'(\sigma))^{-1}]$$
%$$= [f'(\sigma),  |f'|(\alpha), a\lambda'(\sigma)^{-1}]= |f',\lambda'|([\sigma, \alpha,a]).$$
Thus, we get $|f,\lambda| \sim |f',\lambda'|$.
\end{proof}

\begin{lemma}
The functor $|\cdot|\colon \lsimp{}^A \to \ltop{}^A$ is faithful.
\end{lemma}

\begin{proof}
Let $(S,\scrS,l)$ and $(T,\scrT,m)$ be $A$-monomial simplicial complexes. We want to show that for any $(f,\lambda), (f',\lambda') \in \Hom_{\lsimp{}^A}((S,\scrS,l), (T,\scrT,m))$
such that $(|f,\lambda|, |f|)=(|f',\lambda'|, |f'|)$ we have $(f,\lambda)=(f',\lambda')$.
Given $x \in S$, we define a map $\bar{x} \colon S \to I$ such that
\begin{equation*}
	\bar{x}(y):= 
	\begin{cases} 0 & x \neq y, \\ 1 & x=y.\end{cases}
\end{equation*}
Then 
$$f(x)= f(\supp(\bar{x}))= \supp(|f|(\bar{x}))= \supp(|f'|(\bar{x}))= f'(\supp(\bar{x}))= f'(x)$$
for any $x \in S$. Now we will show that $\lambda= \lambda'.$ Let $(\sigma, \alpha, 1) \in E_(S,\scrS,l)$. Then
$$ (f(\sigma), |f|(\alpha), \lambda(\sigma)^{-1})= |f,\lambda|(\sigma,\alpha,1)=|f',\lambda'|(\sigma, \alpha, 1) = (f'(\sigma), |f|(\alpha), \lambda'(\sigma)^{-1}).$$
gives $\lambda(\sigma)= \lambda'(\sigma)$. Thus, we have $(f,\lambda)= (f',\lambda')$.
\end{proof}

%%%%%%%%%%%%%%%%%%%%% SECTION 6 %%%%%%%%%%%%%%%%%%%%%%%%%%%%%%%

%\section{Representations of $G$ in a category $\catC$}

%\begin{nothing}{\it The category $\Rep(G,\catC)$.}
%Let $G$ be a group and let $\catC$ be a category. The category $\Rep(G,\catC)$ of representations of $G$ in $\catC$ has as objects  the pairs $(X, \rho)$, where $X$ is an object of $\catC$ and $\rho \colon G \to \Aut_\catC(X)$ is a group homomorphism. The set of morphisms from $(X,\rho)$ to $(X,\rho')$ consists of all $f\in\Hom_\catC(X,X')$ satisfying $\rho'(g)\circ f = f\circ \rho(g)$ for all $g\in G$. 

%\smallskip
%Each functor $F \colon \catC \to \catD$ induces a functor $\Rep(G,F)\colon\Rep(G,\catC) \to \Rep(G,\catD)$. It maps an object $(X,\rho)$ of $\Rep(G,\catC)$ to $(F(X), F\circ \rho)$ (by post-composing with $F\colon \Aut_\catC(X)\to \Aut_\catD(F(X))$, and which maps a morphism $f \colon (X,\rho) \to (X',\rho')$ in $\Rep(G,\catC)$ to $F(f)\colon F(X)\to F(X')$. Clearly, if $F$ is faithful, so is $\Rep(G,F)$.
%\end{nothing}

%\begin{remark}
%The category of $A$-monomial $G$-posets introduced in \cite{BoucMutlu} and $\Rep(G,\lposet{}^A)$ are isomorphic.
%\end{remark}

%%%%%%%%%%%%% BIBLIOGRAPHY %%%%%%%%%%%%%%%%%%%%%%%%%%%%%%%%%%%%
%\newpage

\end{document}